\documentclass[]{interact}

\usepackage{epstopdf}
\usepackage[caption=false]{subfig}

\usepackage{xcolor}

\usepackage[numbers,sort&compress]{natbib}
\bibpunct[, ]{[}{]}{,}{n}{,}{,}
\usepackage{enumerate}

\theoremstyle{plain}
\newtheorem{theorem}{Theorem}[section]

\newtheorem{corollary}[theorem]{Corollary}
\newtheorem{proposition}[theorem]{Proposition}

\theoremstyle{definition}
\newtheorem{definition}[theorem]{Definition}

\newtheorem{clm}[theorem]{Claim}

\theoremstyle{remark}

\newtheorem{que}{Question}

\newcommand{\eps}{\varepsilon}
\newcommand{\htop}{h_{\text{top}}}
\newcommand{\mdim}{mdim}
\newcommand{\mdiminf}{\underline{mdim}}
\newcommand{\mdimsup}{\overline{mdim}}
\newcommand{\tx}{\tilde{x}}
\newcommand{\ty}{\tilde{y}}
\newcommand{\hp}{\hat{p}}
\newcommand{\Orb}{Orb}
\newcommand{\supp}{supp}

\begin{document}
	
	
	\title{Metric mean dimension of irregular sets for maps with shadowing}
	
	\author{
		\name{Magdalena Fory\'s-Krawiec\textsuperscript{a}\thanks{Magdalena Foryś-Krawiec: maforys@agh.edu.pl, Piotr Oprocha: oprocha@agh.edu.pl} and Piotr Oprocha\textsuperscript{a,b}}
		\affil{\textsuperscript{a} AGH University of Krakow, Faculty of Applied
			Mathematics, al. Mickiewicza 30, 30-059 Krak\'ow, Poland; \textsuperscript{b}National Supercomputing Centre IT4Innovations, University of Ostrava,
			IRAFM,
			30. dubna 22, 70103 Ostrava,
			Czech Republic}
	}
	
	\maketitle
	
	\begin{abstract}
		We study the metric mean dimension of $\Phi$-irregular set $I_{\Phi}(f)$ in dynamical systems with the shadowing property. In particular we prove that for dynamical systems with the shadowing property containing a chain recurrent class $Y$, the values of topological entropy together with the values of lower and upper metric mean dimension of the set $I_{\Phi}(f)\cap B(Y,\eps)\cap CR(f)$ are bounded from below by the respective values for class $Y$.
	\end{abstract}
	

	\section{Introduction}
	
	Metric mean dimension, similarly to the topological entropy, measures the complexity of the dynamical behaviour of a given system. Its definition goes back to Gromov, who proposed it in \cite{Gro}, and to Lindenstrauss and Weiss who developed it further in \cite{LinWei} as an analogue of the topological dimension, combining the dynamical properties of the system together with the metric of the space. 
	
	This paper is an attempt to characterize the structure of irregular sets in dynamical systems with shadowing property over a chain recurrent class in terms of their metric mean dimension. By the classical Birkhoff ergodic theorem we know that the sequence of Birkhoff averages $\frac1n\sum_{i=0}^{n-1}\Phi(f^i(x))$ converges for a set of points with full $f$-invariant measure for a given continuous map $\Phi:X\to \mathbb{R}$. Those points are called \emph{regular points} as the behaviour of the orbits of those points represents the average behaviour in the given system. The remaining points, for which the sequence of Birkhoff averages diverges, are behaving in less regular way, capturing somehow the history of the system. Those points are referred to as \emph{$\Phi$-irregular points} and the set of all such points is called a \emph{$\Phi$-irregular set}, denoted by $I_{\Phi}(f)$. Although the set of irregular points is of measure $0$, it may have an interesting structure and dynamical properties and the orbits of $\Phi$-irregular points may present a complicated behavior, in some cases even as complex as the whole space. For more information about the $\Phi$-irregular sets and their properties in context of their topological entropy and Hausdorff dimension in various dynamical systems, we refer the reader to \cite{BarSch}, \cite{Pes}. Further results on $\Phi$-irregular sets in systems with specification property or with shadowing property may be found in \cite{BarSau}, \cite{DOT}, \cite{FKOT}, \cite{LiWu}, \cite{PesWei}. 
	
	Our research was strongly motivated by the paper of Lima and Varandas \cite{LimVar}, especially their results for the systems admitting the gluing orbit property, which we cite below, changing the notation of \cite{LimVar} slightly to match our setting:
	\begin{theorem}\cite[Theorem E]{LimVar}\label{thm:E}
		Let $f:X\to X$ be a continuous map with the gluing orbit property on compact metric space $X$ and let $\Phi:X\to \mathbb{R}^d$ be a continuous observable. Assume that $I_{\Phi}(f)\neq \emptyset$. Then  $I_{\Phi}(f)$ carries full topological pressure and full metric mean dimension. 
	\end{theorem}
	We say that the map $f:X\to X$ has the \emph{gluing orbit property} if for any $\eps>0$ there is some integer $m>0$ such that for any points $x_1,\dots, x_k \in X$ and any positive integers $n_1,\dots, n_k $ there exists a point $y \in X$ and  some $0\leq p_1\leq\dots\leq p_k\leq m$ with the properties:
	\begin{enumerate}
		\item $d(f^i(y), f^i(x_1))\leq \eps$ for $i=0,\dots, n_1-1$,
		\item $d(f^{i+n_1+p_1+\dots+n_{j-1}+p_{j-1}}(y), f^i(x_j))\leq \eps$ for $i=0,\dots,n_j-1$ and $j=2,\dots,k$.
	\end{enumerate}
	The examples of systems with gluing orbit property but without some stronger properties are known (for example irrational rotations satisfy the gluing orbit property \cite{BTV} but have neither shadowing, nor specification property). On the other hand, all systems with gluing orbit property are transitive. Therefore it is a~closely related, yet not comparable question whether the results from \cite[Theorem E]{LimVar} hold for systems with shadowing property. 
	
	The main result of our paper is the following:
	\begin{theorem}\label{thm:mdim}
		Let $(X,f)$ be a dynamical system with shadowing property, $Y\subset X$ a~chain recurrent class and $\Phi \in \mathcal{C}(X,\mathbb{R})$ a continuous observable. If there exist ergodic measures $\mu_1,\mu_2 \in \mathcal{M}_e(Y)$ such that:
		$$
		\int\Phi d\mu_1 \neq \int\Phi d\mu_2
		$$
		then for any $\eps>0$ the set $I_{\Phi}(f)\cap B(Y,\eps)\cap CR(f)$ is nonempty and the following lower bounds hold:
		\begin{enumerate}[(a)]
			\item $\htop\left(I_{\Phi}(f)\cap B(Y,\eps)\cap CR(f),f\right) \geq \htop(Y,f),$\label{ineq:a}
			\item $\overline{\mdim}_{I_{\Phi}(f)\cap B(Y,\eps)\cap CR(f)}(f) \geq \overline{\mdim}_Y(f),$\label{ineq:b}
			\item $\underline{\mdim}_{I_{\Phi}(f)\cap B(Y,\eps)\cap CR(f)}(f)\geq \underline{\mdim}_Y(f).$\label{ineq:c}
		\end{enumerate}
	\end{theorem}
	
	As the reader may see, in Theorem \ref{thm:mdim} we give the lower bounds for the values of lower and upper metric mean dimension of a particular subset of $\Phi$-irregular set in $X$. We were not able, however, to prove the converse inequality for metric mean dimension, hence we leave the following question open:
	
	\begin{que}
		What additional assumptions on $(X,f)$ are neccesary for  Theorem \ref{thm:mdim} to hold with5 $\overline{mdim}_{I_{\Phi}(f)\cap B(Y,\eps)\cap CR(f)}(f)= \overline{\mdim}_Y(f)$?
	\end{que}

	\section{Preliminaries}
	\subsection{Basic notions and definitions}
	A pair $(X,f)$ consisting of a compact metric space $(X,d)$ and a continuous map $f:X\to X$ is a \emph{dynamical system}. For any point $x \in X$ we define the \emph{orbit} of a point $x$ as the set of all iterations of that point under the map $f$, that is $\Orb(x) = \{f^n(x): \; n\geq 0\}$. The set $\omega_f(x) =  \bigcap_{n\geq 1}\overline{\Orb(f^n(x))} $ is the \emph{$\omega$-limit set of a point $x$}, while the \emph{$\omega$-limit set} is defined as follows: $\omega(f) = \bigcup_{x \in X}\omega_f(x)$. Denote $\mathcal{C}(X,\mathbb{R}) = \{ \varphi:X\to \mathbb{R} : \; \varphi \text{ is continuous}\}$.  Then for dynamical system $(X,f)$
	and a continuous observable $\Phi \in \mathcal{C}(X,\mathbb{R})$ we define the $\Phi$-\emph{irregular set}:
	$$
	I_{\Phi}(f) = \left\{x \in X: \lim_{n\to\infty}\frac1n\sum_{i=0}^{n-1}\Phi(f^i(x)) \text{ diverges}\right\}
	$$
	and the  \emph{irregular set} as the set of all irregular points:
	$$
	I(f) = \bigcup_{\Phi \in \mathcal{C}(X,\mathbb{R})}I_{\Phi}(f).
	$$
	We define the $n$-th Bowen distance between two points $x,y \in X$ as follows:
	$$
	d_n(x,y) = \max\{d(f^i(x),f^i(y)): i=0,\dots,n-1\},
	$$
	and the $(n,\eps)$-Bowen ball $B_n(x,\eps)$ for any $\eps>0$, $n \in \mathbb{N}$ and a point $x \in X$ as the set of points which are $\eps$-close to $x$ according to $n$-th Bowen distance $d_n$, that is:
	$$
	B_n(x,\eps) = \{y \in X: d_n(x,y)<\eps\}.
	$$
	For any set $Z\subset X$ by $B(Z,\eps)\subset X$ we denote the $\eps$-neighborhood of a set $Z$ with respect to the metric $d$. 
	For any Borel set $Z\subset X$, $ n \in \mathbb{N}$ and $\eps>0$ we say that set $E\subset Z$ is an \emph{$(n,\eps)$-separated 
		for $Z$} if $d_n(x,y)>\eps$ for every $x,y \in E$. By $s(Z,n,\eps)$  we denote the maximal cardinality over all $(n,\eps)$-separated subsets for $Z$ in $X$ and we put $s(n,\eps) = s(X,n,\eps)$. Note that 
	the set $Z$ in the above definition does not need to be compact nor $f$-invariant. 
	\begin{definition}
		For any $\delta>0$ a sequence $\{x_n\}_{n \geq 0}\subset X$ is a \emph{$\delta$-pseudo-orbit} of $f$ if $d(f(x_n),x_{n+1})<\delta$ for all $n \in \mathbb{N}_0$. If there exists some $K>0$ such that $x_{n+K} = x_n$ for all $n \geq 0$ we say that $\{x_n\}_{n\geq 0}$ is a \emph{periodic $\delta$-pseudo-orbit}. 
	\end{definition}
	For $0\leq i < j$ a finite $\delta$-pseudo-orbit $x = \{x_n\}_{n=i}^j$ if often referred to as a \emph{$\delta$-chain from $x_i$ to $x_j$} and we denote it by $x_{[i,j]}$. Similarly, if $x\in X$, then we denote $x_{[i,j]}=\{f^i(x),f^{i+1}(x),\ldots, f^j(x)\}$.
	\begin{definition}
		Map $f:X\to X$ satisfies \textit{shadowing property} if for any $\eps>0$ there exists $\delta>0$ such that for any $\delta-$pseudo-orbit $\{x_n\}_{n\geq 0}$ there exists a point $y \in X$ such that $d(x_n, f^n(y))<\eps$ for $n \in \mathbb{N}_0$. 
	\end{definition}
	We say that the point $y$ from the above definition $\eps-$\emph{shadows} the $\delta$-pseudo-orbit $\{x_i\}_{i\geq 0}$ (or simply: \emph{shadows} the $\delta$-pseudo-orbit $\{x_n\}_{n\geq 0}$ if $\eps$ is clear from the context).
	
	For any points $x,y \in X$ and any $\eps>0$ we say that $x$ is in a \textit{chain stable} set of $y$ if there exists a sequence of points $\{x_i\}_{i=0}^n\subset X$ and an increasing sequence of positive integers $\{t_i\}_{i=0}^{n-1}$ such that:
	\begin{align*}
		x_0&=x,\\
		x_n&=y,\\
		d(f^{t_i}(x_i),x_{i+1})&<\eps \text{ for }i=0,\dots,n-1.
	\end{align*}
	
	If $x$ is in the chain stable set of $y$ and $y$ is in the chain stable set of $x$ we say that those points are \emph{chain related}. This relation is an equivalence relation. If $x$ is chain related with itself, we say that $x$ is a \emph{chain recurrent point}. By $CR(f)\subset X$ we denote the set of chain recurrent points in $X$. A \emph{chain recurrent class} is every equivalence class in $CR(f)$ given by the chain relation.
	\begin{definition}
		A subset $Y\subset X$ is \emph{internally chain transitive} if for all $x,y \in Y$ and for every $\eps>0$ there exists an $\eps$-chain from $x$ to $y$ contained in $Y$. 
	\end{definition}
	It is well known that $CR(f)$ is internally chain transitive. Note also that each chain recurrent class $Y\subset X$ is closed and $f$-invariant. 
	
	A point $x \in X$ is \emph{nonwandering} if for every neighborhood $U$ of $x$ there is some $k \in \mathbb{N}$ such that $f^k(U)\cap U \neq \emptyset$. By $\Omega(f)$ we denote the set of all nonwandering points in $X$. Note that $\Omega(f)$ is closed and $f$-invariant. 
	
	\subsection{Measures}
	
	For a dynamical system $(X,f)$  we denote by $\mathcal{M}(X)$ the set of all Borel probability measures on space $(X,\mathcal{B})$, where $\mathcal{B}$ is the $\sigma$-algebra of Borel subsets of $X$. For any measure $\mu \in \mathcal{M}(X)$ the \emph{support} of measure $\mu$, denoted $\supp(\mu)$, is the smallest closed subset $C\subset X$ such that $\mu(C)=1$. A measure $\mu \in \mathcal{M}(X)$ is $f$-invariant if $\mu(f^{-1}(A)) = \mu(A)$ for all $A \in \mathcal{B}$. A measure $\mu \in \mathcal{M}(X)$ is \emph{ergodic} if the only Borel sets satisfying $f^{-1}(B)= B$ are sets of zero or full measure, i. e. sets $B$ with $\mu(B) \in \{0,1\}$. The sets of all $f$-invariant and all ergodic measures are denoted by $\mathcal{M}_f(X)$ and $\mathcal{M}_e(X)$ respectively. Set $\mathcal{M}(X)$ together with the metric given by the weak* topology of the dual space $\mathcal{C}(X,\mathbb{R})$ with the uniform norm can be considered as a compact metric space by the Riesz representation theorem. The convergence in this metric space is defined as follows: the sequence $\{\mu_n\}_{n \in \mathbb{N}}$ converges to a measure $\mu \in \mathcal{M}(X)$ in the weak* topology if the following holds for any $\Phi \in \mathcal{C}(X,\mathbb{R})$:
	$$
	\lim_{n\to\infty}\int\Phi d\mu_n = \int \Phi d\mu.
	$$
	For more information considering measures in the context of dynamical systems we refer the reader to \cite{Den}.
	
	\subsection{Topological entropy and metric mean dimension}
	To start with, we consider the definition of topological entropy for noncompact sets after Bowen given in \cite{Bow}.
	The reader should keep in mind that the values of Bowen topological entropy and the upper capacity topological entropy coincide for invariant and compact spaces. Let $\mathcal{G}_n(Z,\eps)$ be the collection of all finite or countable coverings of the set $Z$ with Bowen balls $B_v(x,\eps)$ for $v\geq n$. Denote:
	$$
	C(Z, t,n,\eps,f) = \inf_{C \in \mathcal{G}_n(Z,\eps)}\sum_{B_v(x,\eps)\in C}e^{-tv}
	$$
	and 
	$$
	C(Z,t,\eps,f) = \lim_{n\to\infty}C(Z, t,n,\eps,f).
	$$
	Define:
	$$
	h_Z(f,\eps) = \inf\{t: C(Z, t,\eps,f)=0 \}= \sup \{ t: C(Z, t,\eps,f)=\infty\}.
	$$
	The \emph{Bowen topological entropy of the set $Z\subset X$} is defined as follows:
	$$
	h_Z(f) = \lim_{\eps\to 0}h_Z(f,\eps).
	$$
	Note that $h_Z(f,\eps)$ is an increasing function of $\eps$. By $\htop(f):=h_X(f)$ we denote \emph{topological entropy} of map $f$. It is well known that the above definition, using Bowen's formula, coincides with the classical definition \cite{Walt}.
	
	An important tool in our further considerations will be the \emph{Katok entropy formula} introduced in \cite{Kat}:
	\begin{equation}\label{eq:KatokEntr}
		h_{\mu}(f) = \lim_{\eps\to 0}\lim_{n\to\infty}\frac1n N_f(n,\eps,\delta),
	\end{equation}
	where $N_f(n,\eps,\delta)$ denotes the smallest number of $(n,\eps)$-Bowen balls covering a subset of $X$ of measure at least $1-\delta$ for some ergodic measure $\mu$. 
	Below we give the definition of lower and upper metric mean dimension of the space $X$ based on $(n,\eps)$-separated sets (after \cite{LinWei}, \cite{LT}): 	
	\begin{definition}\label{def:mdim}
		The \emph{upper metric mean dimension} and \emph{lower metric mean dimension} are defined, respectively, by:
		\begin{eqnarray}
			\mdimsup(f) &=&\limsup_{\eps\to 0}\frac{\limsup_{n\to\infty}\frac1n\log s(n,\eps)}{-\log \eps} \label{df:mdiminf}\\
			\mdiminf(f) &=&\liminf_{\eps\to 0}\frac{\limsup_{n\to\infty}\frac1n\log s(n,\eps)}{-\log \eps} \label{df:mdimsup}
		\end{eqnarray}
	\end{definition}
	If  the above limits are equal, we denote their common value by $\mdim(f)$ and refer to as the \emph{metric mean dimension}.
	We also consider the notion of upper and lower relative metric mean dimension, introduced after \cite{LimVar}:
	\begin{definition} \label{def:Relmdim}
		The \emph{upper relative metric mean dimension of the set $Z$} and  \emph{lower relative metric mean dimension of the set $Z$} are defined, respectively, as follows:
		$$
		\mdimsup_Z(f) =\limsup_{\eps\to 0}\frac{h_Z(f,\eps)}{-\log \eps}, \qquad
		\mdiminf_Z(f) =\liminf_{\eps\to 0}\frac{h_Z(f,\eps)}{-\log \eps}
		$$
	\end{definition}
	If the above limits are equal, we denote their common value by $\mdim_Z(f)$ and refer to as the \emph{relative metric mean dimension of the set $Z$}. Since $X$ is a compact invariant set, we have $\limsup_{n\to\infty}\frac1n s(n,\eps) = h_X(f,\eps)$ (see \cite{PesPit}, \cite{LimVar}), and therefore that $\mdiminf(f) =\mdiminf_X(f) $ and  $\mdimsup(f) =\mdimsup_X(f)$. In other words,  for the case  $Z=X$, notions of dimensions from Definition \ref{def:mdim} and Definition \ref{def:Relmdim} are the same.
	
	The following definition aims at sets $Z$ that share dynamical complexity comparable with the whole space: 
	\begin{definition} Let $Z\subset X$ be $f$-invariant. 
		\begin{enumerate}
			\item $Z$ has full topological entropy if $h_Z(f) = \htop(f)$ . 
			\item $Z$ has full metric mean dimension if:
			$$
			\mdiminf_Z(f) = \mdiminf(f) \text{ and } \mdimsup_Z(f) = \mdimsup(f).
			$$
		\end{enumerate}
	\end{definition}

	\section{Metric mean dimension of irregular sets }
	Below we present a simplified version of the generalized Pressure Distribution Principle proposed by Thompson in \cite[Proposition 3.2]{Thomp}. The notion of topological pressure appearing there, however closely related to the idea of relative metric mean dimension, is beyond the spectrum of the present paper. For more information on topological pressure we refer the reader to \cite{Walt}.
	Note also that Thompson's result does not require convergence of the sequence of probability measures $\{\mu_k\}_{k \in \mathbb{N}}$. In the proof of Theorem \ref{thm:mdim} we use this simplified version (i. e. for constant potential map $\psi=0$): 
	\begin{proposition}\label{prop:Thomp}
		Let $f:X\to X$ be a continuous map on a~compact metric space $X$ and let $Z\subset X$ be a Borel set. Suppose there are $\eps>0$, $s \in \mathbb{R}$, $K>0$ and a sequence of probability measures $\{\mu_k\}_{k \in \mathbb{N}}$ satisfying:
		\begin{enumerate}
			\item $\lim_{k\to\infty}\mu_k = \mu$ for some $\mu \in \mathcal{M}(X)$ such that $\mu(Z)>0$,
			\item  $\limsup_{k\to\infty}\mu_k(B_n(x,\eps))\leq Ke^{-ns}$ for every $n$ large enough and every ball $B_n(x,\eps)$ such that $B_n(x,\eps)\cap Z\neq \emptyset$.
		\end{enumerate} 
		Then $h_Z(f,\eps)\geq s$.
	\end{proposition}
	We are motivated by the following result of \cite{FKOT}, which we present below with some change in the notions to match out setting:
	\begin{theorem}\cite[Theorem 3.1]{FKOT}\label{thm:FKOT}
		Let $(X,T)$ be a dynamical system with shadowing property and let $Y\subset X$ be a chain recurrent class. If $\Phi \in \mathcal{C}(X,\mathbb{R})$ is such that there exist $\mu_1,\mu_2 \in \mathcal{M}_e(Y)$ with:
		$$
		\int \Phi d\mu_1 \neq \int \Phi d\mu_2
		$$
		then $h_{I_{\Phi}(T)}(T)\geq h_Y(T)$.  
	\end{theorem}
	
	We are going to adjust techniques used in the proof of Theorem~\ref{thm:FKOT} so that the growth rate of $h_Z(f,\eps)$ is controlled. This will enable us to estimate the metric mean dimension of Z.
	
	A starting point for the proof of Theorem \ref{thm:FKOT} is the construction of a set $A\subset I_{\Phi}(f)$ with the property that (for $h_Y(f)$ finite): 
	$$h_A(f)\geq h_Y(f)-6\gamma$$
	for some fixed $\gamma>0$ and a chain recurrent class $Y\subset X$. While the argument is not provided directly for $h_Y(f)=\infty$ it follows the same lines, as is explained later. The construction is undertaken with some small $\eps>0$ fixed. 
		For purposes of our proof, we fix $\gamma>0$, arbitrarily large  $n \in \mathbb{N}$ and we will assume that $\eps>0$ is
		small enough, so that:
		$|h_Y(f)-h_Y(f,4\eps)|<\frac{\gamma}{2}$ in case $h_Y(f)<\infty$ and 
		$h_Y(f,4\eps)>n$ in case $h_Y(f)=\infty$. 
 Note that by compactness of $X$, we always have $h_Y(f,4\eps)<\infty$.
	
	The set $A$ is the closure of a set of points which $\eps$-trace some properly constructed $\delta$-pseudo-orbits and it is contained in $B(Y,\eps)$. In order to prove Theorem~\ref{thm:mdim} we need to modify the construction, so that the set $A$ is contained in the set of nonwandering points $\Omega(f)$. Below we present a part of the  proof of Theorem \ref{thm:FKOT} highlighting the most important steps leading to construction of the set $A$ and next, in the proof of Theorem~\ref{thm:mdim}, we introduce necessary adjustments in the construction to obtain a
	sequence of measures enabling us to apply Proposition~\ref{prop:Thomp}. 
	For the full reasoning and more detailed explanation, especially considering the proper choice of some constants, we refer the reader to the proof of \cite[Theorem 3.1]{FKOT}.  From now on we assume that $(X,f)$ is a~dynamical system with shadowing property, $Y\subseteq X$ is a chain recurrent class and $\Phi \in \mathcal{C}(X,\mathbb{R})$ is a~continuous map such that there exist $\mu_1,\mu_2 \in \mathcal{M}_e(Y)$ with: 
	$$\alpha = \int\Phi d\mu_1\neq \int\Phi d\mu_2 = \beta,$$
	assuming without loss of generality (by the Variational Principle) that  $h_{\mu_1}(f)>h_Y(f)-\gamma$  in case $h_Y(f)<\infty$ and $h_{\mu_1}(f)>h_Y(f,4\eps)$ in case $h_Y(f)=\infty$.
		Additionally we assume 
		that $\alpha<\beta$ (the proof in the case $\alpha>\beta$  is the same). 
	The proof in \cite{FKOT} goes in the following steps ((A)-(G) below), which we briefly summarize for the reader convenience. We start by fixing small constant $\eta>0$ which, while fixed, can be chosen to be arbitrarily small.
	Of course, value of $\eta$ will have a direct or indirect consequence on the choice of all other constants and objects.
	We only require that for some fixed $\xi_0\in (0,1)$, sufficiently close to $1$ to satisfy $1-\xi_0<\frac{\gamma}{h_Y(f,4\eps)}$,
	we have $9\eta < (1-\xi_0)(\beta - \alpha)$. 
	As we announced before, let us summarize main steps of the construction in \cite{FKOT}:
	
	\begin{enumerate}[(A)]
		\item Choice of sets $D_1, D_2 \subset Y$ with the property that $\mu_1(D_1)> \frac34, \mu_2(D_2)>\frac34$, consisting of generic points for measures $\mu_1$ and $\mu_2$, respectively, and such that:
		\begin{align} 
			\left|\frac1n\sum_{i=0}^{n-1}\Phi(f^i(x))-\alpha\right|&<\frac{\eta}{4}\text{ for all }x \in D_1\text{ and sufficiently large }n,\label{eq:1}\\
			\left|\frac1n\sum_{i=0}^{n-1}\Phi(f^i(x))-\beta\right|&<\frac{\eta}{4}\text{ for all }x \in D_2\text{ and sufficiently large }n. \label{eq:2}
		\end{align}	
		We also fix an arbitrary $\mu_2$-generic point $y \in D_2$ and a finite open cover $\mathcal{U} = \{U_i\}_{i=1}^S$ of $Y$. 
		
		\item Construction of the family $\{E'_{m_k}\}_{k\geq 0}$ of $(m_k,4\eps)$-separated subsets of  $D_1\subset Y$. The cardinality of each set in the family $\{E'_{m_k}\}_{k\geq 0}$ is not less than $e^{m_k(h_Y(f)-\frac52\gamma)}$ in case $h_Y(f)<\infty$, and not less than $e^{m_k(h_Y(f,4\eps)-\frac{5}{2}\gamma)}$ in case $h_Y(f)=\infty$,  for a properly chosen sequence of integers $\{m_k\}_{k\geq 0}$.The family is obtained by the Katok entropy formula and it depends on particular sets $U,V \in \mathcal{U}$: each set $E'_{m_k}$ contains only those points $x \in U$ for which $f^{m_k}(x) \in V$. 
		
		At this step $\eps>0$ is fixed. It can be arbitrary, provided it is small enough so that the above estimates on the cardinality of each $E'_{m_k}$ from Katok's formula are valid. For this $\eps$ we use shadowing property obtaining $\delta>0$ such that any $\delta$-pseudo-orbit is $\frac{\eps}{4}$-traced.
		For $i,j \in \{1,\dots, S\}$ by $\omega_{ij}$ we denote the length of arbitrarily selected, but fixed for the whole construction, $\delta$-chain between $U_i$ and $U_j$ and put $\omega_{max} =\max_{i,j \in \{1,\dots, S\}}\omega_{ij}$. By chain-transitivity of $f$ on each chain-recurrent class we know
		that such chains exist. 
		
		\item The choice of integer constants $L=m_{k_1}$ and $J>m_{k_2}$ such that $\frac{m_{k_2}}{J} =\xi$ for some $k_1,k_2>0$ and $\xi \in (0,1)$ sufficiently close to $1$, $\xi<\xi_0$ and such that both $\xi J$ and $(1-\xi)J$ are integers. Those constants are necessary to control the length of blocks, which will be used later to build $\delta$-pseudo-orbits. We assume that
		$J>L>\frac{\omega_{max}M}{\eta}$, where $M$ is such that $|\Phi(x)|+\beta\leq M$ for any $x \in X$. We also demand the following conditions to hold:
		\begin{align*}
			\left|\frac{1}{m_k}\sum_{i=0}^{m_k-1}\Phi(f^i(x))-\alpha\right|&<\frac{\eta}{4} \text{ for any }x \in D_1 \text{ and }k>k_1,\\
			\left|\frac{1}{\xi J}\sum_{i=0}^{\xi J-1}\Phi(f^i(x))-\alpha\right| &<\frac{\eta}{4} \text{ for any }x \in D_1,\\
			\left|\frac{1}{(1-\xi)J}\sum_{i=0}^{(1-\xi)J-1}\Phi(f^i(x))-\beta\right|&<\frac{\eta}{4} \text{ for any }x \in D_2.
		\end{align*}
		In other words, $k_1,k_2$ and $\xi$ are such that the above formulas are implied by \eqref{eq:1} and \eqref{eq:2}.  Recall also that for $h_Y(f)<\infty$:
			\begin{align}\label{ineq:E'card}	
				|E'_{m_k}|&>e^{m_k(h_Y(f)-\frac52\gamma)}>e^{m_k(h_Y(f,4\eps-3\gamma))}\text{ for any sufficiently large }k,
			\end{align}
			while for $h_Y(f)=\infty$:
			\begin{align}
				|E'_{m_k}|&>e^{m_k(h_Y(f)-\frac52\gamma)}\text{ for any sufficiently large }k,\nonumber
			\end{align}
			so the inequality (\ref{ineq:E'card}) holds for $\xi J$, $(1-\xi)J$ and $L$ as well.
		\item Choice of $U', V' \in \mathcal{U}$ such that $y \in U'$ and $f^{(1-\xi)J}(y) \in V'$. Then we select $\delta$-chains between the sets $U', V', U,V$ whose lengths by definition are bounded by $\omega_{max}$, that is $P,Q,W<\omega_{max}$:
		\begin{enumerate}
			\item[p)] $\{p_i\}_{i=0}^P$ between $V'$ and $U$,
			\item[q)]$\{q_i\}_{i=0}^Q$ between $V$ and $U$,
			\item[w)]$\{w_i\}_{i=0}^W$ between $V$ and $U'$. 
		\end{enumerate}
		We put $K=J+W$, and obtain, by previous assumptions, that $K>L$ and $K~>~\frac{M\omega_{max}}{\eta}$, which will be useful in estimations later on.
		\item Definition of set $\Gamma_K$ containing all $\delta$-pseudo-orbits built as concatenations of the following blocks:
		\begin{enumerate}[i)]
			\item block of length $\xi J$ of the orbit of some point from $E'_{\xi J}$ (in particular  $\mu_1$-generic),
			\item $\delta$-chain $\{w_i\}_{i=0}^{W-1}$,
			\item block of length $J-m_{k_2} =(1-\xi)J$ of the orbit of the  point $y$ (in particular $\mu_2$-generic).
		\end{enumerate}
		We obtain that for both finite and infinite $h_Y(f)$ (see estimations (3.16)-(3.19) in \cite{FKOT}): 
			\begin{equation}\label{ineq:Gamma}
				|\Gamma_K|\geq e^{J(h_Y(f,4\eps)-3\gamma)}.
		\end{equation}
		\item Construction of the  set of $\delta$-pseudo-orbits $\mathcal{Z} = \{z_i\}_{i \geq 0}$ in $Y$, by cyclical combination of two types of blocks:
		\begin{enumerate}
			\item[(C1)] Blocks of length $L+Q$, that are part of the orbit $x_{[0,L-1]}$ of some $x \in E'_L$ and $\delta$-chain $\{q_i\}_{i=1}^{Q-1}$,
			\item[(C2)] Blocks of length $K+P$, that are some $\delta$-pseudo-orbit from $\Gamma_K$ and $\delta$-chain $\{p_i\}_{i=1}^{P-1}$. 
		\end{enumerate}
		The details of how many repetitions of each block are necessary in the construction may be found in inequalities (3.22) - (3.25) in the original proof in \cite{FKOT}. In particular we want the total length of the blocks (C1) and (C2) to be divisible by the same number in every step of the construction. More formally, we use the blocks of the following form:
		\begin{align*}
			\tx^{(n)} &= x^{(1)}_{[0,L-1]}q_0\dots q_{Q-1} x^{(2)}_{[0,L-1]}q_0\dots q_{Q-1}\dots x^{(l_n\lambda)}_{[0,L-1]}q_0\dots q_{Q-1},\\ 
			\ty^{(n)}  &= y^{(1)}_{[0,K-1]}p_0\dots p_{P-1}y^{(2)}_{[0,K-1]} p_0\dots p_{P-1}\dots y^{(l_n'\kappa)}_{[0,K-1]}p_0\dots p_{P-1},
		\end{align*}
		for some points $x^{(i)} \in E'_L\subseteq U$ where	$i=1,\dots,l_n\lambda$ and $y^{(j)}=(y^{(j)}_s)_{s=0}^{K-1} \in \Gamma_K$ where $j =1,\dots, l_n'\kappa$. 
		We do not require points $x^{(i)}$, nor pseudo-orbits $y^{(j)}=(y^{(j)}_s)$ to be different for different $j$, i.e. we allow repetitions.
		In other words, each sequence $\tx^{(n)}$ is a~concatenation of $l_n\lambda$ blocks of type (C1) while $\ty^{(n)}$ is a~concatenation of $l'_n\kappa$ blocks of type (C2) for some increasing sequences of integers $\{l_n\}_{n \in \mathbb{N}}$, $\{l'_n\}_{n \in \mathbb{N}}$ and sequences $\{a_n\}_{n\in \mathbb{N}}$, $\{b_n\}_{n\in \mathbb{N}}$ defined as follows for $n\in \mathbb{N}$:
		\begin{align*}
			a_1&=0,\\
			b_n&=a_n+M_n,\\
			a_{n+1}&=b_n+M_n' = a_n+M_n+M'_n.
		\end{align*}
		The values $M_n = l_n\lambda(L+Q)$ and $M'_n = l'_n\kappa(K+P)$ determine the length of the blocks $\tx^{(n)}$ and $\ty^{(n)}$ respectively. Moreover, we assume that $\lambda Q \geq \kappa P$ which implies $\lambda L \leq \kappa K$.		 
		The details of the definitions of $M_n$ and $M'_n$ and how they depend on $a_n$ and $b_n$ may be found in formulas (3.23)-(3.25) in \cite{FKOT}.
		The $\delta$-pseudo-orbit $\mathcal{Z} = \{z_i\}_{i \geq 0}$ is defined as follows:
		$$
		z_{[0,M_1-1]} = \tx^{(1)}, \qquad  z_{[M_1,M_1+M'_1-1]} = \ty^{(1)}
		$$
		and, for $n\geq 1$:	
		$$
		z_{[a_n,b_n-1]}= \tx^{(n)}, \qquad  z_{[b_n,a_{n+1}-1]}= \ty^{(n)}.
		$$
		We denote by $\mathcal{D}$ the set of all $\delta$-pseudo-orbits $\mathcal{Z}$ constructed using the above scheme.
		
		\item Definition of the set $A$ as the closure of the set of elements of all possible orbits $\eps$-shadowing all possible $\delta$-pseudo-orbits from $\mathcal{D}$.
		Among consequences of the above construction is the inclusion $A\subset I_{\Phi}(f)$ (see conditions (3.26), (3.28) and (3.29) in \cite{FKOT}).
	\end{enumerate}
	
	Before presenting the proof of Theorem \ref{thm:mdim} let us emphasize that the main difference between the construction of the set $A$ presented above and the one we will provide in the proof below is that we 
	will consider particular sequences of periodic $\delta$-pseudo-orbits obtained by the construction. For each such sequence of periodic $\delta$-pseudo-orbits we will associate a sequence consisting of minimal points, which shadow those periodic $\delta$-pseudo-orbits. That way we obtain a subset of $B(Y,\eps)\cap I_{\Phi}(f)$
	consisting of the limit points of those 
	sequences of minimal points, hence the resulting subset consists of nonwandering points. This vague description will be explained more formally in what follows.
	
	At this point let us recall the following standard procedure which explains utility of periodic pseudo-orbits in the above construction. Consider a periodic $\delta$-pseudo-orbit $\mathcal{Z}=\{z_i\}_{i\geq 0}$, say $z_{s+i}=z_i$ for each $i$.
	If $y$ is a point that $\frac{\eps}{2}$-traces $\mathcal{Z}$ then so is each point $f^{sj}(y)$ for each $j$. But then, by continuity, every point in $\omega$-limit set $q\in C=\omega(y,f^s)$
	satisfies $d(z_i,f^i(q))\leq \frac{\eps}{2}<\eps$. But $C$ as an $f^s$-invariant set contains a minimal subset, so we may assume that $q$ above is a~minimal point.
	This way we obtain a minimal point $\eps$-tracing the pseudo-orbit $\mathcal{Z}$. This fact will be used later without any further reference.
	
	\begin{proof}[Proof of Theorem \ref{thm:mdim}] 
		Let $A$ be the set constructed in the proof of \cite[Theorem 3.1]{FKOT} as summarized in steps (A)-(G) above, with the following modifications. When choosing $L$ fix some small $\tau>0$  and increase, if needed, $L$ in step (C) above in order to satisfy the inequality $L\geq \frac{1-\tau}{\tau}\omega_{max}$. That condition ensures that: 
		\begin{equation}\label{eq:LLowBound}
			L\geq (1-\tau)(L+Q).
		\end{equation}
		When choosing $m_{k_2}$ and $J$, let us increase, if needed, the values from step (C) above in such a way that $J>m_{k_2}\geq \frac{1-\tau}{\tau}2\omega_{max}$, which implies:
		\begin{equation}\label{eq:JLowBound}
			J\geq (1-\tau)(K+P).
		\end{equation} 
		Consequently, the choice of $L$ and $J$ in step (F) together with conditions (\ref{eq:LLowBound}) and (\ref{eq:JLowBound}) assures us that:
		\begin{align}\label{LJBounds}
			\lambda(1-\tau)(L+Q)&=\kappa(1-\tau)(K+P)\leq \lambda L,\\
			\kappa(1-\tau)(K+P)&\leq \kappa J,
		\end{align} 
		where $\lambda, \kappa > 0$ are such that $\lambda(L+Q) = \kappa(K+P)$. Recall that by the construction (step (F)) we have that $\lambda Q \geq \kappa P$ and $K=J+W>L$.
		
		We are going to assure that the set $A$ consist of recurrent points. In order to do so, for any $\delta$-pseudo-orbit $\mathcal{Z} \in \mathcal{D}$ (as constructed in step (F))  
		define a sequence $\{r_n\}_{ n \geq 0}$ of periodic $\delta$-pseudo-orbits in $X$ by putting for each $n \geq 0$ and $j\geq 0$:
		$$
		(r_n)_{[j a_{n+1}, (j+1)a_{n+1}-1]} = \mathcal{Z}_{[0,a_{n+1}-1]}.
		$$
		By the definition of $a_{n+1}$ and the structure of blocks $x^{(n)}$ and $y^{(n)}$ we obtain that the block $(r_n)_{[0,a_{n+1}-1]}$ (and all its repetitions in $r_n$) is a $\delta$-pseudo-orbit starting and ending in the set $U$. Let $\{p_n\}_{n \geq 0}$ be a sequence of minimal points given by the shadowing property which $\eps$-trace the respective $\delta$-pseudo-orbits $r_n$. Let $\{p_{n_k}\}_{k \geq 0}\subset \{p_n\}_{n\geq 0}$ be a~convergent subsequence and denote $\hp_{\mathcal{Z}} = \lim_{k\to \infty}p_{n_k}$. Note that $\hp_{\mathcal{Z}}$ is a nonwandering point for $f$ which additionally is $\frac{\eps}{2}$-tracing $\mathcal{Z}$, since $\delta$ was chosen for $\frac{\eps}{4}$. Now define the set $A$ as the set of all points $\hp_{\mathcal{Z}}$ obtained in the above manner for every possible $\delta$-pseudo-orbit $\mathcal{Z}\in \mathcal{D}$. 
		As we already explained, $A \subset \Omega(f)$. Since the map $f$ has the shadowing property, the set of nonwandering points coincides with the set of chain recurrent points, that is $\Omega(f)=CR(f)$, which in particular implies $A \subset CR(f)$. It is also clear that for any $p \in A$ there exists some $\delta$-pseudo-orbit $\mathcal{Z} = \{z_i\}_{i\geq 0} \in \mathcal{D}$  such that $d(f^i(p),z_i)<\eps$, which in particular implies that $A\subset B(Y,\eps)$.
		
		The set $A$ is a subset of $I_{\Phi}(f)$, which is a straightforward consequence of the following claims, which we recall after \cite{FKOT}.
		\begin{clm}[{\cite[Claim B]{FKOT}}]
			\textit{				For every $u \in A$  
				we have: $$\liminf_{n\to\infty}\frac1n\sum_{i=0}^{n-1}\Phi(f^i(u))\leq \alpha + 4\eta.$$
		}			\end{clm}
		\begin{clm}[{\cite[Claim C]{FKOT}}]
			\textit{				For every $u \in A$ 
				we have: $$\limsup_{n\to\infty}\frac1n\sum_{i=0}^{n-1}\Phi(f^i(u))\geq \zeta -5\eta,$$
				where $ \zeta = \xi \alpha  + (1 - \xi)\beta$. By the choice of $\xi_0$ we have $\zeta-\alpha-9\eta>0$. 
		}			\end{clm}

		For any $\delta$-pseudo obit $\mathcal{Z}=\{z_i\}_{i\geq 0} \in \mathcal{D}$ and any $k \in \mathbb{N}$  consider $\xi^{\mathcal{Z}}_k = \mathcal{Z}_{[0,a_k-1]}$. There exists a point $u \in A$ which $\eps$-shadows $\mathcal{Z}$, hence for $i=0,\dots,a_k-1$ we have $d(f^i(u),z_i)<\eps$. Let $Z_k\subset A$ be the set of points chosen in such a way, that for any $\mathcal{Z} \in \mathcal{D}$ there exists exactly one point $u^{\mathcal{Z}}_k \in Z_k$ such that $u^{\mathcal{Z}}_k$ is $\eps$-shadowing $\xi^{\mathcal{Z}}_k$. 
		Moreover, we assume that each $u\in Z_k$ is $\eps$-shadowing some (unique) $\mathcal{Z}\in \mathcal{D}$ with the prefix $\xi^{\mathcal{Z}}_k$ (we fix one such pseudo-orbit in each step of the construction). 
		This way, we may assume that $Z_k\subset Z_{k+1}$, since when  $\mathcal{Z}\in \mathcal{D}$ is selected at step $k$, point $u$ is shadowing $\mathcal{Z}_{[0,a_s-1]}$ for every $s\geq k$. Therefore, we can extend $Z_k$ to $Z_{k+1}$ in a simple way, constructing shadowing points to those $\xi^{\mathcal{Q}}_{k+1}$, $\mathcal{Q}\in\mathcal{D}$ that are not shadowed by points in $Z_k$. Note also that the construction implies that every part of the $\delta$-pseudo obit of length $a_k$ may be extended to a proper part of the $\delta$-pseudo obit of length $a_{k+1}$ in $\mathcal{D}$ in exactly $|E'_L|^{l_k\lambda}|\Gamma_K|^{l'_k\kappa}$ possible ways, and each of them must be shadowed by a different point,  because by the definition elements of each $E_n'$ are $(n,4\eps)$-separated. This implies that: 
		\begin{equation}\label{eq:Zk}
			|Z_k| = |E'_L|^{l_k\lambda}|\Gamma_K|^{l'_k\kappa}|Z_{k-1}|, k>1.
		\end{equation}	
		Observe that the family $\{B_{a_k}(q,2\eps)\}_{q \in Z_k}$ is a covering of $A$.
		
		Note also that considering estimation (\ref{ineq:E'card}) from (C) on the cardinality of $E_{m_k}$ together with estimation (\ref{ineq:Gamma}) on the cardinality of $\Gamma_k$ from (E) for $m_k = L$ together with the fact that $h_Y(f,4\eps)$ increases when $\eps$ decreases we have:
			\begin{equation}\label{eq:Ecardeps}
				|E'_L| 
				\geq e^{L( h_Y(f,4\eps)-3\gamma)}.
			\end{equation}
			\begin{equation}\label{eq:Gammacardeps}
				|\Gamma_k|
				\geq e^{J(h_Y(f,4\eps)-3\gamma)}
			\end{equation}
		Take $\nu_k = \sum_{q \in Z_k}\delta_q$ and $\mu_k = \frac{1}{|Z_k|}\nu_k$. Observe that for every $k \in \mathbb{N}$ we have:
		$$
		\mu_k(A) = \mu_k(Z_k) = 1.
		$$
		The sequence of measures $\{\mu_k\}_{k\geq 0}$ has a subsequence convergent to some measure $\mu$ and, as $A$ is closed, by standard properties of weak* topology we have $\mu(A)=1$ as well (e.g. see \cite[Proposition 2.7]{Den}). 
		Now fix an arbitrary yet large $n \in \mathbb{N}$  such that for $\tau$ chosen before we have:
		$$
		n-\lambda(L+Q)\geq (1-\tau)n.
		$$
		There is some $k \in \mathbb{N}$ with $a_k<n<a_{k+1}$ . Fix any $r>k$.
		Take $q \in X$ such that $B_n(q,\frac{\eps}{2})\cap A\neq \emptyset$. Then there exists some $u \in B_n(q,\frac{\eps}{2})\cap A$, and, as $\{B_{a_k}(v,2\eps)\}_{v \in Z_k}$ covers $A$, there is a pseudo-orbit $\mathcal{Z}\in \mathcal{D}$ such that $u \in B_{a_k}(u^{\mathcal{Z}}_k,2\eps)$, hence $d_{a_k}(u^{\mathcal{Z}}_k,q)~<~3\eps$. Since $\lambda(L+Q)=\kappa(K+P)$ for every $k \in \mathbb{N}$ we have:
		$$
		a_k = \lambda(L+Q)(l_1+l'_1+\dots+l_k+l'_k).
		$$
		We are going to estimate how many elements of $z\in Z_{k+1}$ satisfy  $d_{a_{k+1}}(z,q)<3\eps$. By the construction we may assume that each such $z$ must be shadowing $\mathcal{Z}_{[0,n-1]}$
		We distinguish two cases, depending on whether the $n$-th point in $\delta$-pseudo obit $\mathcal{Z}$ belongs to the block $\tx^{(k)}$ or $\ty^{(k)}$. \\
		{\bf Case 1. } Assume $a_k < n \leq b_k$. Fix $s \in \mathbb{N}$ such that:
		$$
		a_k+s\lambda(L+Q)<n<a_k+(s+1)\lambda(L+Q).
		$$
		The above implies:
		$$
		a_k+(s-1)\lambda(L+Q)<n-\lambda(L+Q)<a_k+s\lambda(L+Q).
		$$
		
		Note that if we take any point $u \in Z_r \cap B_n(q,\frac{\eps}{2})$ and the associated $\delta$-pseudo-orbit, then its first $n$ elements are uniquely determined by $q$. Therefore by the construction of the $\delta-$pseudo orbits in $\mathcal{D}$ (cf. (\ref{eq:Zk}) ), 
		the necessity of shadowing of the initial segment of $\mathcal{Z}$ and by estimations (\ref{eq:Ecardeps}) and (\ref{eq:Gammacardeps}) we have (for large $r$):
		\begin{align*}
			\mu_r\left(B_n\left(q,\frac{\eps}{2}\right)\right) &= \frac{1}{|Z_r|}\nu_r\left(B_n\left(q,\frac{\eps}{2}\right)\right)\leq \frac{1}{|Z_r|}|E'_L|^{\lambda(l_{k+1}+\dots+l_{r}-s)}|\Gamma_K|^{\kappa(l'_{k+1}+\dots + l'_{r})}	\\
			&\leq |E'_L|^{-\lambda(l_1+\dots + l_k +s)}|\Gamma_K|^{-\kappa(l'_1+\dots +l'_k)}\\
			&\leq e^{-L(h_Y(f,4\eps)-3\gamma  )\lambda(l_1+\dots+l_k+s)}\cdot e^{-J(h_Y(f,4\eps)-3\gamma)\kappa(l'_1+\dots+l'_k)}\\
			&\leq e^{-(h_Y(f,4\eps)-3\gamma)\left[L\lambda(l_1+\dots + l_{ k}
				+s)+J\kappa(l'_1+\dots+l'_k)\right]}\\
			&\leq e^{-(h_Y(f,4\eps)-3\gamma)\left[(1-\tau)(L+Q)\lambda(l_1+\dots + l_{ k}+s)+\kappa(1-\tau)(K+P)(l'_1+\dots+l'_k) \right]}\\
			&\leq e^{-(h_Y(f,4\eps)-3\gamma)(1-\tau)(L+Q)\lambda(l_1+\dots+l_{ k}+s+l_1'+\dots+l'_k)}\\
			&\leq e^{-(h_Y(f,4\eps)-3\gamma)(1-\tau)(n-\lambda(L+Q))}\\
			&\leq e^{-(h_Y(f,4\eps)-3\gamma)(1-\tau)^2n}
		\end{align*}
		
		{\bf Case 2. } Assume that $b_k\leq n<a_{k+1}$. Fix $s' \in \mathbb{N}$ such that:
		$$
		b_{k+1}+(s'-1)\kappa(K+P)<n-\kappa(K+P)<b_{k+1}+s'\kappa(K+P).
		$$
		The above implies:
		$$
		b_{k+1}+(s'-1)\kappa(K+P)\leq n -\kappa(K+P)<b_{k+1}+s'\kappa(K+P).
		$$
		By almost the same argument as in the previous case we get:
		\begin{align*}
			\mu_r\left(B_n\left(q,\frac{\eps}{2}\right)\right) &= \frac{1}{|Z_r|}\nu_r\left(B_n\left(q,\frac{\eps}{2}\right)\right)\leq \frac{1}{|Z_r|}|E'_L|^{\lambda(l_{k+2}+\dots+l_{r})}|\Gamma_K|^{\kappa(l'_{k+1}+\dots + l'_{r}-s')}	\\
			&\leq |E'_L|^{-\lambda(l_1+\dots + l_{k+1}) }|\Gamma_K|^{-\kappa(l'_1+\dots +l'_k+s')}\\
			&\leq e^{-L(h_Y(f,4\eps)-3\gamma  )\lambda(l_1+\dots+l_{k+1})}\cdot e^{-J(h_Y(f,4\eps)-3\gamma)\kappa(l'_1+\dots+l'_k+s')} \\
			&\leq e^{-(h_Y(f,4\eps)-3\gamma)\left[L\lambda(l_1+\dots + l_{k+1})+J\kappa(l'_1+\dots+l'_{k}+s')\right]}\\
			&\leq e^{-(h_Y(f,4\eps)-3\gamma)\left[(1-\tau)\lambda(L+Q)(l_1+\dots + l_{k+1})+(1-\tau)\kappa(K+P)(l'_1+\dots+l'_{k}+s')\right]}\\
			&\leq e^{-(h_Y(f,4\eps)-3\gamma)\left[(1-\tau)\lambda(L+Q)(l_1+\dots + l_{k+1}+l'_1+\dots+l'_{k}+s')\right]}\\
			&\leq e^{-(h_Y(f,4\eps)-3\gamma)(1-\tau)(n-\lambda(L+Q))}\\
			&\leq e^{-(h_Y(f,4\eps)-3\gamma)(1-\tau)^2n}.
		\end{align*}
		
		Altogether, for sufficiently large $r$ we have:
		$$
		\mu_r\left(B_n\left(q,\frac{\eps}{2}\right)\right)\leq e^{-\left(h_Y(f,4\eps)-3\gamma\right)(1-\tau)^2n},
		$$
		hence by the Proposition \ref{prop:Thomp} we get $h_A\left(f,\frac{\eps}{2}\right)\geq \left( h_Y(f,4\eps)-3\gamma\right)(1-\tau)$ and, as $\tau$ is arbitrarily small: 
		$$h_A\left(f,\frac{\eps}{2}\right)\geq h_Y(f,4\eps)-3\gamma.$$
		Considering the topological entropy it follows that:
		$$
		h_{I_{\Phi}(f)\cap CR(f)}\left(f,\frac{\eps}{2}\right) \geq  h_A\left(f,\frac{\eps}{2}\right) \geq h_Y(f,4\eps)-3\gamma,
		$$
		and therefore, since $A\subset  B(Y,\eps)$:
		$$
		h_{I_{\Phi}(f)\cap B(Y,\eps)\cap CR(f)}\left(f,\frac{\eps}{2}\right) \geq h_A\left(f,\frac{\eps}{2}\right)\geq  h_Y(f,4\eps)-3\gamma.
		$$
		Moreover, $\gamma>0$ is arbitrary, hence:
		$$
		h_{I_{\Phi}(f)\cap B(Y,\eps)\cap CR(f)}\left(f,\frac{\eps}{2}\right) \geq h_A\left(f,\frac{\eps}{2}\right) \geq h_Y(f,4\eps).
		$$
	Now, for any fixed $\theta$ and $\eps>\theta$ we have:
			$$
			h_{I_{\Phi}(f)\cap B(Y,\eps)\cap CR(f)}\left(f,\frac{\theta}{2}\right) \geq  h_Y(f,4\theta),
			$$ 
			Taking the limit as $\theta\to 0$ in case $h_Y(f)<\infty$ we get:
			$$
			h_{I_{\Phi}(f)\cap B(Y,\eps)\cap CR(f)}(f)\geq h_Y(f),
			$$
			while in case $h_Y(f)=\infty$ we obtain: $$h_{I_{\Phi}(f)\cap B(Y,\eps)\cap CR(f)}(f)=\lim_{\theta\to 0} h_Y(f,4\theta)=\infty,$$
			which proves the formula in \eqref{ineq:a}. 
		
		For metric mean dimension 
		the reasoning for both infinite and finite $h_Y(f)$ is similar: 
		for $\eps>0$ chosen above and any $\theta \in (0,\eps)$ we get:
		$$
		\frac{h_{I_{\Phi}(f)\cap B(Y,\eps)\cap CR(f)}(f,\frac{\theta}{2})}{-\log \theta}  \geq \frac{  h_Y(f,4\theta) }{-\log \theta}
		$$
		Going with  $\theta \to 0$  to the lower and upper limit we get the formulas \eqref{ineq:b} and \eqref{ineq:c}, respectively :
		\begin{align*}
			\mdiminf_{I_{\Phi}(f)\cap B(Y,\eps)\cap CR(f)}(f) & \geq \mdiminf_Y(f),\\
			\mdimsup _{I_{\Phi}(f)\cap B(Y,\eps)\cap CR(f)}(f) & \geq \mdimsup_Y(f).
		\end{align*}
	\end{proof}
	\begin{corollary}\label{cor:1}
		Let $(X,f)$ be a dynamical system with shadowing property, let $Y\subset X$ be an isolated chain recurrent class and $\Phi \in \mathcal{C}(X,\mathbb{R})$ be a continuous map. If there exist ergodic measures $\mu_1,\mu_2 \in \mathcal{M}_e(Y)$ such that 
		$$
		\int\Phi d\mu_1 \neq \int\Phi d\mu_2,
		$$
		then:
		\begin{enumerate}
			\item $h_{I_{\Phi}(f)\cap Y}(f) = h_Y(f)$,
			\item ${\mdiminf}_{I_{\Phi}(f)\cap Y}(f)= \mdiminf_Y(f)$,
			\item ${\mdimsup}_{I_{\Phi}(f)\cap Y}(f)= \mdimsup_Y(f)$.
		\end{enumerate}  
	\end{corollary}
	\begin{proof}
		By assumption, for sufficiently small $\eps>0$, we have:
		$$I_{\Phi}(f)\cap B(Y,{\eps})\cap CR(f) = I_{\Phi}(f)\cap Y,$$
		and therefore, by Theorem \ref{thm:mdim}:
		\begin{align*}
			h_{I_{\Phi}(f)\cap Y}(f) &\geq h_Y(f)\\
			{\mdiminf}_{I_{\Phi(f)}\cap Y}(f) & \geq \mdiminf_Y(f),\\
			{\mdimsup}_{I_{\Phi(f)}\cap Y}(f) & \geq \mdimsup_Y(f).
		\end{align*}
		The converse inequalities follow directly by definitions. The proof is completed.
	\end{proof}
	
	As the last element, we present a simple consequence of the obtained results. It is also a consequence of Theorem \ref{thm:E} by Lima and Varandas (\cite[Theorem E]{LimVar}). Estimates on entropy were obtained first in \cite{DOT}, and then extended in \cite{FKOT}.
	\begin{corollary}
		Let $(X,f)$ be a transitive dynamical system with shadowing property. Then:
		\begin{enumerate}
			\item $ {h}_{I_{\Phi}}(f)= \htop(f)$,
			\item ${\mdiminf}_{I_{\Phi}(f)}(f)= \mdiminf(f)$,
			\item ${\mdimsup}_{I_{\Phi}(f)}(f)= \mdimsup(f)$.
		\end{enumerate}  
	\end{corollary}
	\begin{proof}
		If entropy of $f$ is zero then the result is obvious. If it is positive, then by standard argument (e.g. see \cite{LO}, cf. \cite{LO2}) there exist ergodic measures $\mu_1,\mu_2 \in \mathcal{M}_e(Y)$ such that 
		$$
		\int\Phi d\mu_1 \neq \int\Phi d\mu_2.
		$$
		Then the result is an immediate consequence of Corollary~\ref{cor:1}, since in our case $X$ is the unique chain-recurrent class.
	\end{proof}

	Despite our attempts to solve the problem, the following question remains open: 
	\begin{que}
		Let $(X,f)$ be a dynamical system with shadowing property. Is it true, that $\mdim_X(f)$ is defined and 
		$$
		\mdim_{I(f)}(f) = \mdim(f)?
		$$
	\end{que}
	
	\section*{Acknowledgements}
	We are grateful to the referees for careful reading and their valuable comments and suggestions that helped to improve this paper and simplify some arguments. 
	
	Piotr Oprocha was supported by National Science Centre, Poland (NCN), grant no. 2019/35/B/ST1/02239.

\end{document}